\documentclass[10 pt]{amsart}
\usepackage{amsfonts,amsmath,amsthm,amssymb}

\usepackage[utf8]{inputenc}
\usepackage{enumerate}
\usepackage{esvect}
\usepackage{mathtools}
\usepackage{bm}
\usepackage[svgnames]{xcolor}
\usepackage{float}
\usepackage{tikz}

\theoremstyle{plain}
\newtheorem{theorem}{Theorem}[section]
\newtheorem{corollary}[theorem]{Corollary}

\theoremstyle{definition}

\theoremstyle{remark}

\newtheorem{example}{Example}[section]

\DeclareMathOperator{\arsinh}{arsinh}
\DeclareMathOperator{\artanh}{artanh} 
\DeclareMathOperator{\sgn}{sgn} 

\begin{document}
\title{Ky Fan inequalities for bivariate means}

\author[A. Witkowski]{Alfred Witkowski}
\email{alfred.witkowski@utp.edu.pl}
\subjclass[2010]{}
\date{\today}

\begin{abstract}
In this note we give sufficient conditions for  bivariate, homogeneous, symmetric means $M$ and $N$ to satisfy Ky Fan inequalities
$$\frac{M}{M'}<\frac{N}{N'}\qquad\text{ and }\quad \frac{1}{M}-\frac{1}{M'}<\frac{1}{N}-\frac{1}{N'}.$$
\end{abstract}


\maketitle

\section{Introduction}
Denote by $\mathsf{G}$ and $\mathsf{A}$ the geometric and arithmetic means of nonnegative numbers $x_i,\ i=1,\dots,n$
$$\mathsf{G}(x_1,\dots,x_n)=\prod_{i=1}^n x_i^{1/n},\quad \mathsf{A}(x_1,\dots,x_n)=\frac{1}{n}\sum_{i=1}^n x_i.$$

The classical Ky Fan inequality states that if $0<x_i\leq 1/2$, then
$$\frac{\mathsf{G}(x_1,\dots,x_n)}{\mathsf{G}(1-x_1,\dots,1-x_n)}\leq \frac{\mathsf{A}(x_1,\dots,x_n)}{\mathsf{A}(1-x_1,\dots,1-x_n)}.$$

This result has been published in 1961 and generalized in many directions. For more information see \cite{Alzer1995}.

This note is inspired by the result of Neuman and S\'andor \cite{NS2003Fa} where we find the following sequence of inequalities 
\begin{equation}
\frac{\mathsf{G}}{\mathsf{G'}}\leq \frac{\mathsf{L}}{\mathsf{L'}}\leq \frac{\mathsf{P}}{\mathsf{P'}}\leq \frac{\mathsf{A}}{\mathsf{A'}}\leq \frac{\mathsf{NS}}{\mathsf{NS'}}\leq \frac{\mathsf{T}}{\mathsf{T'}},
\label{eq:chain1}
\end{equation}
(here ${\mathsf{L}}, {\mathsf{P}}, {\mathsf{NS}}, {\mathsf{T}}$ stand for the logarithmic, first Seiffert, Neuman-S\'andor and second Seiffert means of arguments $0<x,y\leq 1/2$, and prime denotes the same mean with arguments $1-x$ and $1-y$).

The aim of this note is to establish the conditions for two symmetric, homogeneous means $M, N$ under which  the Ky Fan inequality
$$\frac{M(x,y)}{M(1-x,1-y)}\leq \frac{N(x,y)}{N(1-x,1-y)} $$
holds.
\section{Definition and notation}
A function $M:\mathbb{R}_+^2\to\mathbb{R}$ is called a \textit{mean} if for all $x,y>0$ holds
$$\min(x,y)\leq M(x,y) \leq \max(x,y).$$
Mean is \textit{symmetric} if for all $M(x,y)=M(y,x)$ for all $x,y>0$ and \textit{homogeneous} if for all $\lambda>0$ 
$$M(\lambda x,\lambda y)=\lambda M(x,y).$$
A function $m:(0,1)\to\mathbb{R}$ satisfying
$$\frac{z}{1+z}\leq m(z)\leq \frac{z}{1-z}$$
is called a \textit{Seiffert function}.

Let us recall a result from \cite{Witkowski}
\begin{theorem}\label{thm:repr}
The formula	
$$M(x,y)=\frac{|x-y|}{2m\left(\frac{|x-y|}{x+y}\right)}$$
establishes a one-to-one correspondence between the set of symmetric, homogeneous means and the set of Seiffert functions.
\end{theorem}
If there is no risk of ambiguity we shall skip the argument of means. For $0<x,y\leq 1/2$ and an arbitrary mean $M$ we shall denote $M(1-x,1-y)$ by $M'$.

\section{Ky Fan inequalities}
Our first theorem shows that the Ky Fan inequality is a consequence of monotonicity of the ratio of Seiffert means. 
\begin{theorem}\label{thm:1}
	Let $M,N$ be two symmetric, homogeneous means and $m,n$ be their Seiffert functions. If the function $\frac{n}{m}$ is decreasing, then for all $0<x,y\leq 1/2$ the Ky Fan inequality
	$$\frac{M(x,y)}{M(1-x,1-y)}\leq \frac{N(x,y)}{N(1-x,1-y)} $$
	holds.
\end{theorem}
\begin{proof}
	Denote $x'=1-x, y'=1-y$. Clearly $|x-y|=|x'-y'|$ and  $x+y\leq x'+y'$, thus $\frac{|x-y|}{x+y}\geq\frac{|x'-y'|}{x'+y'}$. By Theorem \ref{thm:repr} we obtain
	$$\frac{M(x,y)}{N(x,y)}=\frac{n}{m}\left(\frac{|x-y|}{x+y}\right)<\frac{n}{m}\left(\frac{|x'-y'|}{x'+y'}\right)=\frac{M(x',y')}{N(x',y')},$$
	which concludes the proof.
\end{proof}
Sometimes it is more convenient to investigate the ratio of means instead of their Seiffert functions.
\begin{corollary}\label{corr:1}
	In the assumptions of Theorem \ref{thm:1} if the function $q(t)=\frac{M(1,t)}{N(1,t)}$ increases in $(0,1)$, then for all $0<x,y\leq 1/2$ the Ky Fan inequality
	$$\frac{M(x,y)}{M(1-x,1-y)}\leq \frac{N(x,y)}{N(1-x,1-y)} $$
	holds.
\end{corollary}
\begin{proof}
	Let us recall the formula connecting the mean and its Seiffert function (\cite{Witkowski})
	\begin{equation}
	m(z)=\frac{z}{M(1+z,1-z)}.
	\label{eq:Mto m}
	\end{equation}
	This gives:
	$$\frac{n}{m}(z)=\frac{M(1+z,1-z)}{N(1+z,1-z)}=\frac{M\left(1,\frac{1-z}{1+z}\right)}{N\left(1,\frac{1-z}{1+z}\right)}=\frac{M(1,s)}{N(1,s)}=q(s).$$
	Here $s=\frac{1-z}{1+z}$ decreases from $1$ to $0$ as $z$ travels in the opposite direction, so $n/m$ decreases if and only if $q$ increases.
\end{proof}
 
Let us illustrate our results with some examples:
\begin{example}
	This result was proved by Chan, Goldberg and Gonek in \cite{CGG1974}: if 
	$$\mathsf{A}_r(x,y)=\begin{cases}
		 \left(\frac{x^r+y^r}{2}\right)^{1/r}	 & r\neq 0,\\
		\sqrt{xy} & r=0
	\end{cases}$$
	is a power mean of order $r$, then for $r<s$ holds
	\begin{equation}
	\frac{\mathsf{A}_r}{\mathsf{A}_r'}\leq\frac{\mathsf{A}_s}{\mathsf{A}_s'}.
	\label{neq:ar<as}
	\end{equation}
	
	Indeed, for $0<t<1$ and $rs\neq 0$
	$$\sgn\frac{d}{dt}\log\frac{\mathsf{A}_r(1,t)}{\mathsf{A}_s(1,t)}=\sgn\frac{t^r-t^s}{t(t^r+1)(t^s+1)}=\sgn(s-r),$$
	else if $s=0$, then
		$$\sgn \frac{d}{dt}\log\frac{\mathsf{A}_r(1,t)}{\mathsf{A}_0(1,t)}=\sgn \frac{t^r-1}{t(2t^r+1)}=-\sgn r,$$
	so the inequality \eqref{neq:ar<as} is true by Corollary \ref{corr:1}.
\end{example}
\begin{example}
	The last term in the chain of inequalities \eqref{eq:chain1} contains the second Seiffert mean
	$$\mathsf{T}(x,y)=\frac{|x-y|}{2\arctan\frac{|x-y|}{x+y}}.$$
	Let $\mathsf{Q}(x,y)=\sqrt{\frac{x^2+y^2}{2}}$ be the quadratic (or root mean square). \\We have $\mathsf{t}(z)=\arctan z$ and $\mathsf{q}(z)=\frac{z}{\sqrt{1+z^2}}$ and
	
		$$\frac{d}{dz}\frac{\mathsf{t}}{\mathsf{q}}(z)=\frac{z-\arctan z}{z^2\sqrt{1+z^2}}>0,$$
so by Theorem \ref{thm:1} we can extend the chain \eqref{eq:chain1}
$$\frac{\mathsf{T}}{\mathsf{T'}}<\frac{\mathsf{Q}}{\mathsf{Q'}}.$$
\end{example}
\begin{example}
	Consider now the Heronian mean
	$$\mathsf{He}(x,y)=\frac{x+\sqrt{xy}+y}{3}.$$
	For $0<t<1$ we have 
	$$q(t)=\frac{\mathsf{He}(1,t)}{\mathsf{A}_{2/3}(1,t)}=\frac{2\sqrt2}{3}\frac{t+\sqrt{t}+1}{(t^{2/3}+1)^{3/2}}$$
	and
	$$\frac{dq}{dt}(t)=\frac{\sqrt{2}}{3}\frac{(1-t^{1/6})^3(1+t^{1/6})}{(1+t^{2/3})^{5/2}t^{1/2}}>0.$$
	We also see that the quotient
	\begin{align*}
	\frac{\mathsf{He}(1,t)}{\mathsf{A}_{1/2}(1,t)}	& =\frac{4}{3}\frac{t+\sqrt{t}+1}{t+2\sqrt{t}+1}=\frac{4}{3}\left[1-\frac{\sqrt{t}}{t+2\sqrt{t}+1}\right]\\%
	&=\frac{4}{3}\left[1-\frac{1}{\sqrt{t}+2+\frac{1}{\sqrt{t}}}\right]
	\end{align*}
	decreases in $(0,1)$.
	Therefore by Corollary \ref{corr:1} we have 
	$$\frac{\mathsf{A}_{1/2}}{\mathsf{A}_{1/2}'}<\frac{\mathsf{He}}{\mathsf{He}'}<\frac{\mathsf{A}_{2/3}}{\mathsf{A}_{2/3}'}.$$
\end{example}
\begin{example}
	We know (\cite{Lin1974F}) that the logarithmic mean $\mathsf{L}(x,y)=\frac{x-y}{\log x-\log y}$ satisfies the inequality $\mathsf{L}<\mathsf{A}_{1/3}$. Consider
	$$q(t)=\frac{\mathsf{A}_{1/3}(1,t)}{\mathsf{L}(1,t)}=\frac{(t^{1/3}+1)^3}{8(t-1)}\log t.$$
	Its derivative equals
	$$\frac{dq}{dt}(t)=\frac{(t^{1/3} + 1)^2 [t^{4/3} + t - t^{1/3} - 1- (t + t^{1/3}) \log t ]}{8 t(t - 1)^2 }.$$
	To evaluate the sign of the expression in square brackets substitute $t=s^3$ and calculate its Taylor series at $s=1$
	\begin{align*}
		t^{4/3} + t &- t^{1/3}- 1 - (t + t^{1/3}) \log t =s^4+s^3 -s-1-3(s^2+1)s\log s\\%
		&=-\sum_{n=5}^\infty \frac{n^2-5n+12}{n(n-1)(n-2)(n-3)}(1-s)^n<0,
	\end{align*}
	which proves the inequality
	$$\frac{\mathsf{L}}{\mathsf{L}'}<\frac{\mathsf{A}_{1/3}}{\mathsf{A}_{1/3}'}.$$
\end{example}
\begin{example}There is another mean similar to the Seiffert means that lies between the arithmetic and the first Seiffert mean:

$$\mathsf{P}(x,y)=\frac{|x-y|}{2\arcsin\left(\frac{|x-y|}{x+y}\right)} \leq \frac{|x-y|}{2\sinh\left(\frac{|x-y|}{x+y}\right)}=\mathsf{S}_{\sinh}(x,y)\leq \frac{x+y}{2}=\mathsf{A}(x,y).$$

Their respective Seiffert functions are $\arcsin, \sinh$ and $\mathrm{id}$.
The hyperbolic sine is convex for positive arguments, therefore $\frac{\sinh z}{z}$ increases as a divided difference of a convex function. So by Theorem \ref{thm:1} we conclude
$$\frac{\mathsf{S}_{\sinh}}{\mathsf{S}_{\sinh}'}\leq \frac{\mathsf{A}}{\mathsf{A}'}.$$
\end{example}
\begin{example}
	It is known \cite{Jag1994F} that $\mathsf{A}_{1/2}<\mathsf{P}$. We have $\mathsf{a}_{1/2}(z)=\frac{2z}{1+\sqrt{1-z^2}}$. To check the monotonicity of $\mathsf{p}/\mathsf{a}_{1/2}$ we substitute $z=\sin t$ to get $\frac{t}{2\sin t}(1+\cos t)$. Its derivative $\frac{t-\sin t}{2\cos t-2}$ is negative, so
	$$\frac{\mathsf{A}_{1/2}}{\mathsf{A}_{1/2}'}<\frac{\mathsf{P}}{\mathsf{P}'}.$$
\end{example}
\begin{example}
	The tangent is also a Seiffert function. It satisfies the inequality $\tan z<\artanh z$ for $0<z<1$, so their means satisfy
	$$\mathsf{L}(x,y)=\frac{|x-y|}{2\artanh\left(\frac{|x-y|}{x+y}\right)}<\frac{|x-y|}{2\tan\left(\frac{|x-y|}{x+y}\right)}=\mathsf{S}_{\tan}(x,y).$$
	Let us investigate the quotient of the Seiffert functions:
	$$\frac{d}{dz}\frac{\artanh z}{\tan z}=\frac{\frac{1}{2}\sin 2z-(1-z^2)\artanh z}{(1-z^2)\sin^2z}$$
	and using Taylor expansion we obtain
	
	\begin{align*}
	\frac{1}{2}\sin 2z-&(1-z^2)\artanh z	 =\frac{1}{2}\sum_{n=0}^\infty \frac{(-1)^n (2z)^{2n+1}}{(2n+1)!}-(1-z^2)\sum_{n=0}^\infty \frac{ z^{2n+1}}{2n+1}\\%
	&=\sum_{n=1}^\infty\frac{1}{(2n+1)(2n-1)}\left[2+\frac{(-1)^n 4^n(2n-1)}{(2n)!}\right]z^{2n+1}>0,
\end{align*}
because all the coefficients in square brackets are nonnegative. Thus by Theorem \ref{thm:1}
$$\frac{\mathsf{L}}{\mathsf{L}'}<\frac{\mathsf{S}_{\tan}}{\mathsf{S}_{\tan}'}.$$
\end{example}

\section{Harmonic Ky Fan inequalities}
In  \cite{NS2006} the autors derive the inequalities of type
$$\frac{1}{\mathsf{A}}-\frac{1}{\mathsf{A}'}<\frac{1}{\mathsf{P}}-\frac{1}{\mathsf{P}'}   $$
from the inequalities $\mathsf{P}/\mathsf{P}'<\mathsf{A}/\mathsf{A}'$ using some elementary algebraic transformation. Here we show an approach that bases on the concept of Seiffert functions.

Suppose $M$ ans $N$ are two symmetric and homogeneous means. Using Theorem \ref{thm:repr} the inequality
$$\frac{1}{M}-\frac{1}{M'}\leq \frac{1}{N}-\frac{1}{N'}$$
can be written as
$$m\left(\frac{|x-y|}{x+y}\right)-n\left(\frac{|x-y|}{x+y}\right)\leq m\left(\frac{|x-y|}{2-x-y}\right)-n\left(\frac{|x-y|}{2-x-y}\right).$$
Since $2-x-y>x+y$ for $0<x,y<1/2$ we can formulate the following result.
\begin{theorem}\label{thm:2}
	Suppose the means $M$ and $N$ are symmetric and homogeneous and $m$ and $n$ are their Seiffert means. If the function $m-n$ decreases, then the Ky Fan inequality
	$$\frac{1}{M}-\frac{1}{M'}\leq \frac{1}{N}-\frac{1}{N'}$$
	holds.
\end{theorem}
Taking into account the formula \eqref{eq:Mto m} and setting $s=(1+z)/(1-z)$ we get the following
\begin{corollary}
	In the assumptions of Theorem \ref{thm:2} if the function 
	$$g(s)=(s-1)\left(\frac{1}{M(s,1)}-\frac{1}{N(s,1)}\right)$$
	descreases for $s>1$, then the inequality
	$$\frac{1}{M}-\frac{1}{M'}\leq \frac{1}{N}-\frac{1}{N'}$$
	holds.
\end{corollary}
\begin{example}
Consider the chain of inequalities between Seiffert functions (see \cite[Lemma 3.1]{Witkowski}) valid for $0<z<1$
$$z>\arsinh z>\sin z>\arctan z>\tanh z.$$
The mean corresponding to the second function is called Neuman-S\'andor mean --- $\mathsf{NS}$ and the fourth one generates the second Seiffert mean --- $\mathsf{T}$.

	Since $\cosh^2z-1=(\cosh z+1)(\cosh z-1)>2\cdot \frac{z^2}{2}$
	$$\frac{d}{dz}(\arctan z-\tanh z)=\frac{\cosh^2z-1-z^2}{(1+z^2)\cosh^2z}>0.$$
	Inequality $\cos z>1-z^2/2$ leads to
	$$ \frac{d}{dz}(\sin z-\arctan z)=\frac{(1+z^2)\cos z-1}{1+z^2}>\frac{(1+z^2)(1-z^2/2)-1}{1+z^2}=\frac{z^2(1-z^2)}{2(1+z^2)}>0.$$
	On the other hand $\cos z<1-z^2/2+z^4/6$ , so
\begin{align*}
\frac{d}{dz}(\arsinh z-\sin z)	& =\frac{1-\sqrt{1+z^2}\cos z}{\sqrt{1+z^2}}>\frac{1-\sqrt{(1+z^2)\left(1-\frac{z^2}{2}+\frac{z^4}{6}\right)^2}}{\sqrt{1+z^2}}\\%
&=\frac{1-\sqrt{1-\frac{5}{12}z^4(1-z^2)-\frac{1}{36}z^8(5-z^2)}}{\sqrt{1+z^2}}>0.
\end{align*}
And finally
	$$\frac{d}{dz}(z-\arsinh z)=1-\frac{1}{\sqrt{1+z^2}}>0.$$
	Therefore by Theorem \ref{thm:2}
		$$\frac{1}{\mathsf{S}_{\tanh}}-\frac{1}{\mathsf{S}_{\tanh}'}\leq \frac{1}{\mathsf{T}}-\frac{1}{\mathsf{T}'}\leq \frac{1}{\mathsf{S}_{\sin}}-\frac{1}{\mathsf{S}_{\sin}'}\leq \frac{1}{\mathsf{NS}}-\frac{1}{\mathsf{NS}'}\leq \frac{1}{\mathsf{A}}-\frac{1}{\mathsf{A}'}.$$
\end{example}
\begin{example}
	On the other side of the arithmetic mean there are two chains of inequalities for Seiffert means involving sine and tangent (see \cite[Lemma 3.2]{Witkowski}):
	\begin{equation}
	z<\sinh z<\genfrac{\{}{\}}{0pt}{}{\tan z}{\arcsin z}<\artanh z.
	\label{neq:chain1}
	\end{equation}
\end{example}
The two Seiffert functions in curly brackets are not comparable, arcsine defines the first Seiffert mean and inverse hyperbolic tangent corresponds to the logarithmic mean.

The difference $\sinh z-z$ increases, because this is the gap between a convex function and its supporting line.

To show that 
$\frac{d}{dz}(\tan z-\sinh z)=\frac{1-\cos^2z\cosh z}{\cos^2z}$ is positive note that $$1-\cos^2z\cosh z>1-\cos z\cosh z=:f(z).$$
The function  $f$ vanishes at $z=0$ and
$$f'(z)=\cos z\cosh z(\tan z-\tanh z)>0,$$
because $\tan z>z>\tanh z$. Thus $f'$ is positive, and so is $f$.

The difference between inverse hyperbolic tangent and tangent also increases, since
$$\frac{d}{dz}(\artanh z-\tan z)'=\frac{1}{1-z^2}-\frac{1}{\cos^2 z}=\frac{z^2-\sin^2z}{(1-z^2)\cos^2z}>0.$$

To deal with the lower chain of inequalities \eqref{neq:chain1} note that
\begin{align}\label{neq:cosh}
\cosh z=&1+\frac{z^2}{2!}+\frac{z^4}{4!}+\frac{z^6}{6!}+\dots\\\notag
<&1+\frac{z^2}{2}+\frac{z^4}{12}\left(\frac{1}{2}+\frac{1}{4}+\dots\right)=1+\frac{z^2}{2}+\frac{z^4}{12}.
\end{align}

Now
$$\frac{d}{dz}(\arcsin z-\sinh z)=\frac{1-\sqrt{1-z^2}\cosh z}{\sqrt{1-z^2}}>0$$
because using \eqref{neq:cosh}
\begin{align*}
	1-\sqrt{1-z^2}\cosh z&>1-\sqrt{(1-z^2)\left(1+\frac{z^2}{2}+\frac{z^4}{12}\right)^2} \\%
	&=1-\sqrt{1-\frac{7z^4}{12}-\frac{z^6}{3}-\frac{11z^8}{144}-\frac{z^{10}}{144}}>0.
\end{align*}

Comparison of the last pair is simple
$$\frac{d}{dz}(\artanh z-\arcsin z)=\frac{1}{1-z^2}-\frac{1}{\sqrt{1-z^2}}>0.$$

Now we can use Theorem \ref{thm:2} to write the chain of inequalities

$$\frac{1}{\mathsf{A}}-\frac{1}{\mathsf{A}'}<\frac{1}{\mathsf{S}_{\sinh}}-\frac{1}{\mathsf{S}_{\sinh}'}<
\genfrac{\{}{\}}{0pt}{}{\frac{1}{\mathsf{S}_{\tan}}-\frac{1}{\mathsf{S}_{\tan}'}}{\frac{1}{\mathsf{P}}-\frac{1}{\mathsf{P}'}}
<\frac{1}{\mathsf{L}}-\frac{1}{\mathsf{L}'}.$$

NOTE: In the proof of Theorems \ref{thm:1}  we use monotonicity of $\frac{m}{n}$ to obtain inequality between $\frac{m}{n}\left(\frac{y-x}{x+y}\right)$ and $\frac{m}{n}\left(\frac{y-x}{2-x-y}\right)$. It is natural to ask  whether the inverse is true, i.e. if for all $0<x<y<1/2$ the inequality
\begin{equation}
f\left(\frac{y-x}{x+y}\right)>f\left(\frac{y-x}{2-x-y}\right)
\label{neq:conjecture}
\end{equation}
holds then $f$ needs to be monotone. The counterexample has been produced  by \textit{tometomek91} - user of the mathematical portal \texttt{matematyka.pl}. \\
It is clear that $\frac{y-x}{x+y}>\frac{y-x}{2-x-y}$. Moreover $\sup_{0<x<y<1/2}\frac{y-x}{2-x-y}=\frac{1}{3}$. Therefore every function $f$ satisfying the two conditions:
\begin{itemize}
	\item $f$ is increasing on $(0,1/3)$
	\item $f(x)\geq \lim_{t\nearrow{1/3}}f(t)$ for $x\in [1/3,1]$
\end{itemize}
satisfies \eqref{neq:conjecture}.

\section{Acknowledgement} The author wishes to thank Janusz Matkowski and Monika Nowicka for improvement suggestions and efficient bug and typo hunting.

\section{References}
%

\end{document}